\documentclass[12pt]{amsart}
\usepackage{amssymb,amsmath,amsthm,amsfonts,mathrsfs}
\usepackage{enumitem}
\usepackage[margin=2cm]{geometry}
\usepackage{graphicx}
\usepackage[dvipsnames]{xcolor}
\usepackage{xspace}
\usepackage{comment}
\usepackage{hyperref}

\newcommand{\stp}{\,\,\,}

\newcommand{\demph}[1]{\emph{#1}}

\newcommand{\mcal}[1]{\mathcal{#1}}

\newcommand{\vphi}{\varphi}
\renewcommand{\vec}[1]{\mathbf{#1}}

\newcommand{\eqdef}{:=}
\newcommand{\COS}{dioid\xspace}
\newcommand{\mCOS}{\ensuremath{\mathfrak{m}}-\COS}
\newcommand{\m}{\ensuremath{\mathfrak{m}}}

\newtheorem{thm}{Theorem}[section]
\newtheorem{crl}[thm]{Corollary}
\newtheorem{lem}[thm]{Lemma}
\newtheorem{prop}[thm]{Proposition}

\newtheorem{puz}{Puzzle}

\theoremstyle{definition}
\newtheorem{defi}[thm]{Definition}
\newtheorem*{defi*}{Definition}

\theoremstyle{remark}
\newtheorem{eg}[thm]{Example}
\newtheorem{rk}[thm]{Remark}

\newcommand{\st}[2]{\left\{#1\,:\,#2\right\}}
\newcommand{\pset}[1]{\mathcal{P}(#1)}
\newcommand{\psetr}[1]{\mathcal{P}_1(#1)}
\newcommand{\crd}[1]{\left|#1\right|}
\newcommand{\union}{\cup}
\newcommand{\Union}{\bigcup}
\newcommand{\ffam}{\ensuremath{\mathcal{F}}}
\newcommand{\flr}[1]{\lfloor #1 \rfloor}
\newcommand{\B}{\mathbb{B}}
\newcommand{\N}{\mathbb{N}}
\newcommand{\Z}{\mathbb{Z}}
\newcommand{\Q}{\mathbb{Q}}
\newcommand{\R}{\mathbb{R}}
\newcommand{\C}{\mathbb{C}}
\newcommand{\F}{\mathbb{F}}

\newenvironment{mx}{\left(\begin{matrix}}{\end{matrix}\right)}
\DeclareMathOperator{\per}{per}

\newcommand{\cun}{\mathrel{\ooalign{$\cup$\cr
  \hidewidth\raise.25ex\hbox{{\scriptsize$\times\mkern1.413mu$}}\cr}}}

\begin{document}
\title[Linear independence over $\mathfrak{m}$-dioids]{Linear independence over naturally-ordered semirings with applications to dimension arguments in extremal combinatorics}
\author{Gal Gross}
\address{Department of Mathematics, University of Toronto,
Toronto, ON, Canada M5S 2E4}
\email{g.gross@mail.utoronto.ca}
\date{\today}

\begin{abstract}
    A family of subsets $\ffam \subseteq \pset{\{1, 2, \ldots, n\}}$ has the \emph{disparate union property} if any two disjoint subfamilies $\ffam_1, \ffam_2 \subseteq \ffam$ have distinct unions $\Union \ffam_1 \neq \Union \ffam_2$; what is the maximal size of a family with the disparate union property? Is there a simple and efficiently computable characterization of size-maximal families?

    This paper highlights a class of partially-ordered semirings---difference ordered semirings with a multiplicatively absorbing element---and shows it is common and easily constructed. We prove that a suitably modified definition of linear independence for semimodules over such semirings enjoys the same maximality property as for vector spaces, and can furthermore be efficiently detected by the bideterminant. These properties allow us to extend dimension argument in extremal combinatorics and provide simple and direct solutions to the puzzles above. 
\end{abstract}



\maketitle

\section{Motivation}
\label{sec:motivation}
A family of subsets $\ffam \subseteq \pset{\{1, 2, \ldots, n\}}$ has the \demph{distinct union property} if any two distinct subfamilies $\ffam_1, \ffam_2 \subseteq \ffam$ have distinct unions $\Union \ffam_1 \neq \Union \ffam_2$. (The union of the empty family is taken to be the emptyset.) For example, the family of singletons $\ffam = \{\{1\}, \{2\}, \ldots, \{n\}\}$ has the distinct union property. It is an elementary exercise to prove that this is the \emph{only} family of size $\geq n$ that has the distinct union property, and so it is the unique maximal family with this property. (The proof appears at the end of the section so as not to spoil the puzzle.) Let us therefore relax the restriction: a family of subsets $\ffam \subseteq \pset{\{1, 2, \ldots, n\}}$ has the \emph{disparate union property} if any two disjoint subfamilies $\ffam_1, \ffam_2 \subseteq \ffam$ (not both empty) have distinct unions $\Union \ffam_1 \neq \Union \ffam_2$. Again, the family of singletons $\ffam = \{\{1\}, \{2\}, \ldots, \{n\}\}$ has the disparate union property, but it is no longer the only example. Each of $\ffam' = \{\{1\}, \{1, 2\}, \{1, 2, 3\}, \ldots, \{1, 2, \ldots, n\}\}$ and $\ffam'' = \{\{1\}, \{1, 2\}, \{1, 3\}, \ldots, \{1, n\}\}$ is a family of size $n$ with the disparate union property. In fact, each of these examples is maximal in the sense that no set can be added without destroying the property. Are these examples size-optimal?

\begin{quote}
    \begin{puz}
    \label{puz:howlarge}
    What is the largest possible size $m$ of a family $\ffam\subseteq \pset{\{1, 2, \ldots, n\}}$ with the disparate union property?
    \end{puz}
\end{quote}

These kind of questions are typical of extremal combinatorics. Let $[n] \eqdef \{1, 2, 3, \ldots, n\}$, a subset $\ffam\subseteq \pset{[n]}$ is called a \demph{family} over $[n]$ (or an \demph{$n$-family}, also known as a \demph{set-system} over $[n]$). The prototypical extremal combinatorics question is: 
\begin{quote}
    \textbf{Question}: Given some property $P$, what is the largest possible $n$-family satisfying $P$?
\end{quote}

Some famous examples include:
\begin{enumerate}[label=(\roman*)]
    \item For $P_{Sp}$ the property that ``no set is a (strict) subset of another'' (i.e., an $\subset$-antichain, also called a \demph{clutter}), Sperner's Lemma \cite{Spe28} states that ${n\choose \flr{n/2}}$ is the maximum.
    \item For $P_{EKR}$ the property that ``every set has size $k$ and they are pairwise disjoint'', the Erd\H{o}s-Ko-Rado Theorem \cite{EKR61} states that if $k\leq n/2$ then ${n-1\choose k-1}$ is the maximum.
    \item\label{oddtown} For $P_{2}$ the property that ``every set has an odd cardinality and every intersection of two distinct sets has even cardinality'', the Oddtown Theorem \cite{Ber69} states that $n$ is the maximum.
    \item\label{mod6} For $P_6$ the property that ``every set has cardinality not divisible by $6$, and every intersection of two distinct sets has cardinality divisible by $6$'', the maximum size is not known, but a 1998 upper bound attributed to M.~Szegedy \cite[Ex.~1.1.28]{BF22} is $2n-2\log 2n$.
    \item The property $P_u$ that ``no two disjoint subfamilies have the same union'' is studied in this paper.
\end{enumerate}

Extremal problems such as the above have been solved through a variety of methods, as surveyed by Anderson \cite{And87} and more recently by Jukna \cite{Juk11}, with monographs dedicated to probabilistic methods \cite{AS15}, polynomial methods \cite{Gut16}, and linear algebraic methods \cite{BF22}, among others. It is this last which concerns us at present.

T.~Gowers \cite{Gow08} (later also posted as \cite{Tri08}) proposed a heuristic for  combinatorial questions amenable to linear algebraic methods: when there seems to be many different maximal examples, look for a \emph{dimension argument}. That is, represent the objects as vectors and use the property that a list of linearly independent vectors in an $n$-dimensional vector space has size at most $n$  (for an overview of different linear-algebraic methods in combinatorics see \cite[chs.~13--14]{Juk11}). We propose to solve Puzzle \ref{puz:howlarge} by precisely such a method, and in Section \ref{subsec:unpack} we present a short solution using linear algebra. 

The key idea behind dimension arguments is to represent each set $S \in \ffam$ by its incidence vector $\vec{v}_S$ with $1$ in position $i$ if and only if $i \in S$

\[(\vec{v}_S)_i = \begin{cases}
    1 &\mbox{ if } i \in S,\\
    0 &\mbox{ if } i \notin S.
\end{cases}\]

One then works in a vector space whose operations correspond to the restrictions imposed by $P$ in a natural way. For example, recounting Berlekamp's proof \cite{Ber69}: working in $\F_2^n$, the restrictions $P_2$ from \ref{oddtown} are translated to the conditions on incidence vectors: 
\begin{itemize}
    \item A given set has odd cardinality if and only if the inner-product of the corresponding incidence vector with itself is nonzero.
    \item Two sets have intersection of even cardinality if and only if the inner-product of the corresponding incidence vectors is zero.
\end{itemize}
Taken together, these conditions imply that the set of incidence vectors is linearly independent in $\F_2^n$, and so its size cannot exceed $n$. Dimension arguments are  appealing in their simplicity and directness. One important limitation is the need to work over a (possibly finite) field. For example, the argument fails to generalize in a straightforward way to $P_6$ from \ref{mod6} above, since $\Z/6\Z$ is not a field.

For the problem at hand, involving the condition $P_u$, after representing the sets in $\ffam$ by incidence vectors, we would like a linear dependence of the form
\[\lambda_1\vec{v}_1+\lambda_2\vec{v}_2+\cdots+\lambda_k\vec{v}_k = \mu_1\vec{u}_1 + \mu_2\vec{u}_2 + \cdots + \mu_j\vec{u}_j\]
to indicate that the union of the sets corresponding to the $\vec{v}_i$'s is the same as the union of the sets corresponding to the $\vec{u}_i$'s. However, the union of sets is an idempotent operation $S \union S = S$, whereas field operations are cancellative: $1+1 \neq 1$. This is a fundamental limitation of working with incidence vectors over a field. 

We propose to work over an algebraic structure where the ``addition'' operation is idempotent, and where the linear-algebraic argument above works precisely---namely, a semimodule over a particular kind of semiring. Many ideas from linear algebra carry over to the setting of semimodules over semirings, as is surveyed by Golan \cite{Gol99} in great detail and generality, and by Gondran and Minoux \cite{GM08} for the particular (but very common) case of naturally-ordered semirings (also called difference-ordered semirings and canonically-ordered semirings); see also G\l{}azek's guide \cite{Gla02} for an exhaustive bibliography up to 2001. However, equivalent notions of linear independence in a vector space over a field diffract into many inequivalent notions in semimodules over semirings. Moreover, even semimodules over difference-ordered semirings is too large a class in general, they do not all enjoy convenient analogues of linear-algebraic properties. We introduce a particular class of difference-ordered semirings---those which possess a multiplicatively-absorbing element---and show they are easily constructed and especially suitable for questions involving families of sets. We then prove that this class enjoys nice linear-algebraic properties which allow us to solve Puzzle \ref{puz:howlarge} in a straightforward manner.

In addition to being of independent interest, this extension not only presents a straightforward solution of the puzzle, which can be easily ``translated back'' to linear algebra, but also allows us to characterize all maximal solutions with a simple determinant-like criterion that is nevertheless not easy to ``translate'' into linear algebra.\vspace{-0.8cm}

\begin{quote}
    \begin{puz}
    \label{puz:criterion} Given a family of $\pset{\{1, 2, \ldots, n\}}$ of size $m$, is there an efficient way to determine whether it has the disparate union property?
    \end{puz}
\end{quote}

We conclude this section with an overview of the remainder of the paper. In Section \ref{sec:COS} we introduce canonically-ordered semirings (\emph{dioids}) and highlight a particular kind (\emph{\mCOS}) which is central to our investigation; to my knowledge, this particular kind has not been singled out for study as a class before. In Section \ref{sec:lin} we introduce semimodules over semirings, present one particular notion of linear independence, and show that it allows for dimension arguments over \mCOS{s}, this allows us to solve Puzzle \ref{puz:howlarge} by a direct argument as indicated above; subsection \ref{subsec:unpack} ``unpacks'' the solution and ``translates'' it to linear algebra. In Section \ref{sec:bid} we introduce the bideterminant, which is the semimodule-generalization of the determinant, and prove it can be used to detect linear independence in semimodules over \mCOS{s}, which leads to a solution of Puzzle \ref{puz:criterion}.

Readers familiar with difference-order semirings can start at Section \ref{sec:lin}. For those wishing to skip directly to the novel contributions, they are Lemma \ref{lem:forward}, Theorem \ref{thm:bool}, Theorem \ref{thm:toR}, Theorem \ref{thm:fromN}, a simple linear-algebraic proof of the known Theorem \ref{thm:boolbid}\footnote{This result has been proved before \cite{GM77,GM78,GM84} via an ingenious graph-theoretic argument for the larger class of selective-invertible dioids, of which tropical rings---also known as max-plus algebras or max-algebras---are a notable example.}, and Corollary \ref{crl:mbid}. Puzzles \ref{puz:howlarge} and \ref{puz:criterion} are also novel and their solutions are presented in Corollaries \ref{crl:solhowlarge} and \ref{crl:solcrit}.%

\begin{proof}[Solution to the distinct union property puzzle.]
Let $\ffam \subseteq \pset{\{1, 2, \ldots, n\}}$ be a family of size $\geq n$ with the distinct union property. Note that $\emptyset \notin \ffam$, for otherwise we may choose $\ffam_1 = \{\emptyset\}$ and $\ffam_2 = \emptyset$. There are $2^{\crd{\ffam}}\geq 2^{n}$ subfamilies (including the empty family), and if each has distinct union then $\crd{\ffam} \leq n$ and these unions are precisely the $2^n$ subsets of $\{1, 2, \ldots, n\}$. In particular, for each singleton $\{a\}$, there is some subfamily $\ffam_a \subseteq \ffam$ with $\Union\ffam_a = \{a\}$, so that $\ffam_a = \{\{a\}\}$ (since $\emptyset \notin \ffam$). This proves that $\{\{1\}, \{2\}, \ldots, \{n\}\} \subseteq \ffam$, and since $\crd{\ffam} \leq n$ we conclude that $\ffam = \{\{1\}, \{2\}, \ldots, \{n\}\}$.
\end{proof}

\section{Naturally-ordered semirings}\label{sec:COS}
Recall that a field is often conceptualized as the melding of two groups, whereas a ring is conceptualized as a melding of a group and a monoid, i.e., the requirement of multiplicative inverses is dropped. A semiring is a further generalization where the requirement of additive inverses is also dropped, and can be conceptualized as the melding of two monoids. The definitions in this section are standard unless noted otherwise, the lemmas all follow directly from the definition and proofs are omitted. The reader is referred to Golan's monograph \cite{Gol99} for an extensive overview of semirings and their applications.

\begin{defi}[Semiring]
    A \emph{semiring} is a triple $(S, +, \times)$ with $S$ a nonempty set and $+, \times: S\times S \to S$ binary operations such that:
\begin{enumerate}[label=(\alph*)]
    \item $(S, +)$ is a commutative monoid with identity $0$. That is, $+$ is associative, commutative, and for every $s \in S$ we have $s+0 = 0+s = 0$.
    \item $(S, \times)$ is a monoid (not necessarily commutative) with two-sided identity $1$. That is, $\times$ is associative and for every $s \in S$ we have $1\times s = s\times 1 = s$.
    \item Multiplication (left- and right-)distributes over addition: for any $a, b, c \in S$,
        \begin{align*}
            &a\times (b+c) = (a\times b) + (a\times c),\\
            &(a+b)\times c = (a\times c) + (b\times c).
        \end{align*}
        \item\label{ax:0} The additive identity is multiplicatively absorbing: for any $a \in M$, $a\times 0 = 0\times a = 0$.
\end{enumerate}
The semiring $(S, +, \times)$ is said to be \emph{commutative} exactly when $(S, \times)$ is a commutative monoid. 
\end{defi}

This paper focuses on one kind of semirings, which is also an ordered semiring, or more precisely \demph{difference ordered semiring} in Golan's terminology \cite[p.~228]{Gol99}; following Gondran and Minoux \cite{GM08}, we call difference ordered semirings \demph{dioids}\footnote{Golan points out \cite[p.~3]{Gol99} that both Baccelli et al.~\cite{Bac+92} and Gunawardena \cite{Gun98} have used that same term ``dioid'' to denote additively-idempotent semirings. 

Recall that an element $m$ of a monoid $(M, +)$ is \demph{idempotent} if $m+m=m$, and the monoid $(M, +)$ is \demph{idempotent} if each of its elements is idempotent. Any idempotent commutative monoid $(M, +)$ is canonically-ordered {\cite[ch.~1]{GM08}}. The Gondran and Minoux's sense of ``dioid'' can thus be seen as a generalization of the previous use of the term, so we hope it shall not cause undue confusion. 
}. In their monograph on this algebraic structure, Gondran and Minoux demonstrate that it is particularly rich and suitable to model many problems from graph theory and optimization; they conclude the preface to their book \cite{GM08} with: ``One of the ambitions of this book is thus to show that, as complements to usual algebra, based on the construct ``Group-Ring-Field'', other algebraic structures based on alternative constructs, such as ``Canonically ordered monoid-dioid-distributive lattice'' are equally interesting and rich, both in terms of mathematical properties and of applications.''

\begin{defi}[Canonical order {\cite[ch.~1]{GM08}}]
    Let $(M, +)$ be a commutative monoid and define the binary relation $\leq$ on $M$ as follows: for any $a, b \in M$, we say that $a \leq b$ if and only if there exists some $c \in M$ such that $a+c = b$. Then $\leq$ is a preorder (i.e., reflexive and transitive) on $M$.

    We call $(M, +)$ \demph{canonically ordered} just in case that the canonical preorder is a partial order (i.e., is also antisymmetric).
\end{defi}

\begin{rk}
    Note that for any \emph{group} $(M, +)$ the canonical preorder is trivial: $a\leq b$ for any $a, b \in M$. Conversely, if $(M, +)$ is canonically-ordered then, other than $0$, no element has an inverse. Thus, canonically ordered monoids and groups (and consequently rings and dioids) lie on opposite ends of a spectrum, see Gondran and Minoux \cite{GM08} for further discussion.
\end{rk}

Note that canonically-ordered monoids enjoy the zerosumfree property.

\begin{defi}[Zerosumfree {\cite[p.~4]{Gol99}}]
    A commutative $(M, +)$ monoid is \demph{zerosumfree} if for every $a, b \in M$: $a+b = 0$ implies $a = b = 0$.
\end{defi}

\begin{lem}\label{lem:zerosum}
    Any canonically-ordered monoid $(M, +)$ is zerosumfree.
\end{lem}

Semirings whose additive monoid is canonically ordered are called \demph{dioids}.

\begin{defi}
    A \demph{dioid} (also known as \demph{difference ordered semiring}, \demph{naturally ordered semiring}, and \demph{canonically-ordered semiring}) is a semiring $(M, +, \times)$ such that the monoid $(M, +)$ is canonically-ordered. It is said to be \demph{commutative} just in case it is commutative as a semiring (i.e., if $(M, \times)$ is commutative).
\end{defi}

\begin{rk}
The canonical order $\leq$ associated with $(M, +)$ is defined entirely in terms of addition, and it is important to note that it respects multiplication \cite[ch.~1]{GM08}: suppose $a \leq b$, then for any $m \in M$ we have $a\times m \leq b\times m$.
\end{rk}

Dioids are natural structures which abound throughout mathematics. Gondran and Minoux \cite{GM08} give many examples. Two intensely-studied classes include:
\begin{itemize}
    \item Distributive lattices $L$ with a minimum $0$ and a maximum $1$ form a commutative \COS $(L, \vee, \wedge)$ under the join and meet operations. As a concrete example, for any set $A$, the triple $(\pset{A}, \union, \cap)$ is a commutative \COS with $0 = \emptyset$ and $1 = A$. 
    \item The \demph{tropical semirings} $(\R\union\{-\infty\}, \max, +)$ and $(\R\union\{+\infty\}, \min, +)$ are each a commutative \COS. (The reader is referred  to Butkovi\v{c}'s monograph \cite{But10} for an extensive overview of max-algebras and their applications.)
\end{itemize}

An element $a$ of a commutative monoid $(M, +)$ is \demph{absorbing} if $a+m = a$ for every $m \in M$; Golan \cite{Gol99} calls such an element \demph{infinite}. If the monoid is canonically-ordered, such an element (if it exists) is necessarily the $\leq$-greatest element. Conversely, for canonically-ordered monoids, the $\leq$-greatest element $a$ is also additively absorbing: this is because $a+m$ is greater than $a$, and if $a$ is greatest then antisymmetry implies that $a+m = a$. To help distinguish it from multiplicatively-absorbing elements later on, we propose to name this element the \demph{top element}.

\begin{lem}
    Any finite nonempty canonically-ordered monoid $(M, +)$ has top element $\top = \sum_{m \in M}m$.
\end{lem}

\begin{proof}
    For any $a \in M$ we have $a+\sum_{m \in M\setminus\{a\}}m = \top$ (the empty-sum is interpreted as $0$), which proves $a \leq \top$.
\end{proof}

The example of $(\pset{A}, \union, \cap)$ illustrates that an additively-absorbing element ($\top = A$) need not be multiplicative-absorbing (indeed, in that example it is the multiplicative identity, which is as far from multiplicatively absorbing as possible). On the other hand, by the semiring axioms, $0$ is always multiplicatively absorbing. Semirings with nonzero multiplicatively absorbing element $\m\neq 0$ are our main focus in this paper. As far as I am aware, \mCOS as a class have never been singled out for investigation, and never in the context of linear algebra (i.e., semimodules).

\begin{defi}\label{def:mCOS}
    An \demph{\mCOS} is a commutative \COS $(M, +, \times)$ which has a nonzero multiplicatively absorbing element; that is, an element $0 \neq \m \in M$ such that $a\times \m = \m$ for every nonzero $0 \neq a \in M$.
\end{defi}

If such an element $\m$ exists then it is unique since if $a \in M$ is another nonzero multiplicatively absorbing element then $a = a\times \m = \m$. 

LaGrassa \cite[Example 1.1]{LaG95} shows that any commutative monoid $(M, \times)$ can be embedded in an \mCOS as follows (adapting the construction to our notation): first, add a new absorbing element $\m$. Then, define $a+b = \m$ for every $a, b \in M\union\{\m\}$. Finally, add $0$ (which is both the additive identity and multiplicatively absorbing). LaGrassa \cite{LaG95} studies the ideals of such a semiring in relation to the monoid ideals of the monoid $(M, \times)$ from which it was constructed. 

LaGrassa's construction already presents a large collection of examples. We now prove that the powerset construction is another rich source of examples; the powerset construction is especially suited for questions regarding families of sets.

\begin{prop}[{\cite[Example 1.10]{Gol99}}]
    Let $(M, \times)$ be a commutative monoid (with identity $1_M$). The \demph{power-\COS} of $M$, denoted $2^{M}$, is the triple $(\pset{M}, \union, \otimes)$ with $\otimes$ defined pointwise:
    \[A\otimes B \eqdef \st{a\times b}{a\in A, b\in B}.\]
    The power-\COS is an idempotent \COS with $0 = \emptyset$, $1 = \{1_M\}$, and top element $\top = M$. Moreover, $2^M$ is commutative if and only if $(M, \times)$ is commutative.
\end{prop}

\begin{proof}
    It is clear that $(\pset{M}, \union)$ is an idempotent monoid with $0 = \emptyset$, and hence canonically-ordered. It is also easy to see that $(\pset{M}, \times)$ is a monoid with $1 = \{1_M\}$. Distributivity follows from the pointwise definitions:
    \begin{align*}
        A\otimes(B\union C)
        &\eqdef \st{a\times x}{x \in B\union C}\\
        &= \st{a\times b}{a \in A, b \in B} \union \st{a\times c}{a\in A, c \in C}\\
        &= (A\otimes B)\union (A\otimes C). 
    \end{align*}
    Completely analogously $(A\union B)\otimes C = (A\otimes C)\union (B\otimes C)$.

    The pointwise definition also implies that $\emptyset$ is multiplicatively absorbing: $\emptyset \otimes A = A\otimes \emptyset = 0$. This proves that $2^M$ is a \COS. It is immediate that $\top = M$ is a top element.

    Finally, if $(M, \times)$ is commutative, so is $(2^M, \otimes)$. Moreover, $(M, \times)$ is embedded in $(2^M, \otimes)$ as the collection of singletons: $a \mapsto \{a\}$, so that if $(M, \times)$ is not commutative, then neither is $(2^M, \otimes)$.
\end{proof}

The power-\COS has a non-trivial addition and is naturally suited to investigations of families from $\pset{M}$. For any finite commutative monoid $(M, \times)$ the power-\COS is in fact an \mCOS.

\begin{lem}\label{lem:powerdioid}
    Let $(M, \times)$ be a commutative monoid with top element ${\top_{\!M}}$. Then $2^M$ is an \mCOS with $\m = \{\top_{\!M}\}$.
\end{lem}

\begin{proof}
    Note that $\m = \{\top_{\!M}\} \neq \emptyset$ so it is different from $0$. For any nonzero $A \in 2^M$ we have 
    \[A\otimes \m = \st{a\times \top_{\!M}}{a \in A} = \{\top_{\!M}\} = \m\]
    since $\top_{\!M}$ is absorbing in $(M, \times)$.
\end{proof}

\begin{eg}
    The commutative monoid $(\N, +)$ has no top element. The corresponding power-\COS $(\pset{\N}, \union, +)$ is not an \mCOS. Indeed, any nonempty $A \subseteq \N$ has some minimum element $a \in A$, but $\{a+1\}+A$ consists entirely of natural numbers strictly greater than $a$ so that $\{a+1\}+A \neq A$.
\end{eg}

Even in cases where $(M, \times)$ is a commutative monoid with no top element, as in the previous example, a restriction of the power-set construction still yields an \mCOS.

\begin{prop}
    Let $(M, \times)$ be a commutative monoid (with identity $1_M$). Define
    \[\psetr{M} \eqdef \st{A \subseteq M}{1_M \in A} \union \{\emptyset\}.\]
    Then the restriction of $2^M$ to $\psetr{M}$ is an \mCOS with $\top = \m = M$. We call the elements of $\psetr{M}$ \demph{rooted sets}, and the triple $(\psetr{M}, \union, \otimes)$ the \demph{rooted \mCOS} of $M$.
\end{prop}

\begin{proof}
    It is clear that $\psetr{M}$ forms a sub-\COS of $2^M$. For any nonzero element $A \in \psetr{M}$ we have $M\subseteq M\otimes A$ (since $1 \in A$), and since $X \subseteq M$ for any $X \in  \psetr{M}$ we conclude that $M\otimes A = M$.
\end{proof}

(In this example, the top element coincides with $\m$. A semiring element which is both additively-absorbing and multiplicative absorbing is called \demph{strongly infinite} by Golan \cite[p.~5]{Gol99}.)

\begin{eg}
    For the commutative monoid $(\N, +)$, we have $\psetr{M}$ is the (emptyset and the) collection of all subsets of $\N$ which have $0$ as an element.\footnote{Combinatorial questions regarding which rooted and unrooted subsets of $\{0, 1, 2, \ldots, k\}$ have the most divisors---with respect to the dioid product, i.e., the Minkowski sum---were previously investigated by the author \cite{Gro20}, proving some conjectures of Applegate, LeBrun, and Sloane \cite{ABS11}.}
\end{eg}

In addition to appearing as natural constructions, \mCOS{s} enjoy several nice properties as semirings. LaGrassa \cite{LaG95} noted that the semiring embedding of a multiplicative monoid is zerosumfree and entire. These properties hold in any \mCOS. We have already proved the zerosumfree property in Lemma \ref{lem:zerosum}, which holds in any canonically-ordered monoid. 

\begin{defi}{{\cite{Gol99}}}
    Let $(M, +, \times)$ be a semiring. A nonzero element $0 \neq z \in M$ is called a \demph{zero-divisor} if there exists some nonzero $0 \neq b \in M$ with $bz = 0$. The semiring $(M, +, \times)$ is called \demph{entire} just in case it has no zero-divisors.
\end{defi}

\begin{lem}
    Any \mCOS is entire.\footnote{Golan \cite[p.~154]{Gol99} calls a zerosumfree and entire semimodule (see Section \ref{sec:lin}) an \demph{information semimodule} and proves that a semiring is zerosumfree and entire if and only if it admits a nontrivial information semimodule. Golan references Takahashi's paper \cite{Tak85} for a classification of cyclic information semimodules over $\N$, though Takahashi does not refer to these semimodules by any specific name and I have not been able to find other references to ``information semimodules''.}
\end{lem}

\begin{proof}
    Let $(M, +, \times)$ be an \mCOS, and assume for contradiction that $0\neq z \in M$ is a zero-divisor. Then, there exists some $b \in M$ with $0 = bz$ so that
    \[0 = 0\times \m = (bz)\times \m = b\times(z\times\m) = b\times\m = \m\]
    a contradiction.
\end{proof}

The simplest possible semiring is the $2$-element \demph{Boolean semiring} $\B = \{0, 1\}$ which has only additive and multiplicative identities, and where $1+1 = 1$\footnote{Equivalently, $\B = (\{0, 1\}, \max, \min)$. Note that the $2$-element \demph{Boolean algebra} also has a complementation operation $\overline{1} = 0$ and $\overline{0} = 1$; in contrast, the 2-element \demph{Boolean ring} is $\F_2$ so that $1+1 = 0$.}. Semiring morphisms are defined in the expected manner; Golan \cite[p.~107]{Gol99} calls a semiring morphism $M \to \B$ a \demph{character} of $M$.

\begin{defi}[{\cite[ch.~9]{Gol99}}]
    Let $(M, +_M, \times_M)$ and $(N, +_N, \times_N)$ be semirings. A function $\phi: M \to N$ is called a \demph{morphism (of semirings)} if:
    \begin{enumerate}[label=(\roman*)]
        \item $\phi$ preserves $0$'s: $\phi(0_M) = 0_N$;
        \item $\phi$ preserves $1$'s: $\phi(1_M) = 1_N$;
        \item $\phi$ is additive: $\phi(a+_Mb) = \phi(a)+_N\phi(b)$, for every $a, b \in M$; 
        \item $\phi$ is multiplicative: $\phi(a\times_M b) = \phi(a)\times_N\phi(b)$, for every $a, b \in M$.
    \end{enumerate}
\end{defi}

We now define a Boolean morphism which we call the characteristic Boolean morphism, as a nod to both ``character'' and ``characteristic function.'' This morphism will play an important role in our investigation of linear independence in \mCOS.

\begin{lem}[{\cite[9.10]{Gol99}}]
\label{lem:boolmorph}
    Let $(M, +, \times)$ be a semiring and define $\vphi_M: M \to \B$ by
    \[\vphi_M(a) = \begin{cases}
        0 &\mbox{ if } a = 0,\\
        1 &\mbox{ if } a \neq 0.
    \end{cases}\]
    Then $\vphi_M$ is a morphisms of semirings if and only if $M$ is zerosumfree and entire; in such cases, we call $\vphi_M$ the \demph{characteristic Boolean morphism}.
\end{lem}

Since semirings do not have inverses, failure of injectivity $\phi(a) = \phi(b)$ does not necessarily imply $\phi(m) = 0$ for some nonzero $m \neq 0$. We therefore make the following definition (novel, to my knowledge).

\begin{defi}
    A semiring morphism $\phi: (M, +_M, \times_M) \to (N, +_N, \times_N)$ is called \demph{entire} if $(a\neq 0_M) \implies (\phi(a) \neq 0_N)$.
\end{defi}

The characteristic Boolean morphism is necessarily entire.

\section{Linear independence over \mCOS}\label{sec:lin}
\label{sec:LI}
Semimodules over semirings directly generalize modules over rings.

\begin{defi}[Semimodule]
    A \demph{semimodule} over a commutative semiring $(M, +, \times)$ is a triple $(V, +, \cdot)$ such that:
    \begin{enumerate}[label=(\roman*)]
        \item $(V, +)$ is a commutative monoid with identity $\vec{0}$.
        \item The scalar multiplication $\cdot$ is a function  $M\times V \to V$ such that:
        \begin{enumerate}[label=(\alph*)]
            \item scalar multiplication distributes over vector addition: for any $\lambda \in M$ and $\vec{u, v} \in V$,
            \[\lambda\cdot(\vec{u}+\vec{v}) = (\lambda\cdot\vec{u}) + (\lambda\cdot\vec{v}).\]
            \item semiring addition ``distributes'' over scalar multiplication: for any $\lambda, \mu \in M$ and $\vec{v} \in V$,
            \[(\lambda+\mu)\cdot\vec{v} = (\lambda\cdot\vec{v})+(\mu\cdot\vec{v}).\]
            \item scalar multiplication is compatible with semiring multiplication: for every $\lambda, \mu \in M$ and $\vec{v} \in V$,
            \[(\lambda \times \mu)\vec{v} = \lambda\cdot(\mu\cdot\vec{v})\]
            \item identity and zero for scalar multiplication: for any $\vec{v} \in V$,
            \begin{align*}
                &1\vec{v} = \vec{v},\\
                &0\vec{v} = \vec{0}.
            \end{align*}
            \item the zero vector is absorbing: for any $\lambda \in M$,
            \[\lambda\cdot\vec{0} = \vec{0}.\]
        \end{enumerate}
    \end{enumerate}
\end{defi}

(The last two conditions provably hold for any vector space, but must be included as axioms for semimodules.) We adopt the linear-algebra conventions of denoting elements of $M$ (\demph{scalars}) by Greek letters and elements of $V$ (\demph{vectors}) by bold letters. We also write $\lambda\vec{v}$ for the scalar multiplication $\lambda\cdot\vec{v}$.

\begin{eg}
For any commutative semiring $(M, +, \times)$ and any positive integer $n$, the triple $(M^n, +, \cdot)$ with operations defined componentwise is a semimodule over $(M, +, \times)$, as expected; we call this the \demph{standard $n$-semimodule} over $M$. In the case that $M$ is a dioid, the result is furthermore a moduloid.
\end{eg}

\begin{defi}[{\cite[p.~174]{GM08}}]
    A \demph{moduloid} (also called \demph{naturally ordered semimodule} or \demph{canonically ordered semimodule}) is a semimodule $(V, +, \cdot)$ over a \COS $(M, +, \times)$ such that the commutative monoid $(V, +)$ is canonically-ordered.
\end{defi}

Notions of linear independence which are equivalent for a vector space over a field become distinct when scalars no longer have additive or multiplicative inverses. Generalizations of linear independence of varying strengths have been investigated in the literature, one of the earliest (\demph{weak linear independence}) is due to Cuninghame-Green \cite{CG79} (see \cite[p.~178]{GM08} and \cite[p.~191]{Gol99} for brief overviews and references for different notions of linear independence); most recently a new concept of \demph{tropical independence} has been proposed and studied, we refer the reader to the discussion in \cite{AGG09} for this (and related) concept of linear independence in tropical semirings. In this section we focus on an especially useful generalization, introduced by Gondran and Minoux \cite{GM77,GM78,GM84} who give it the name ``linear independence''. We follow Golan's \cite[p.~191]{Gol99} more agnostic terminology and call this concept ``linear attachment''.

Many arguments for semimodules over a \COS proceed by somehow separating the ``positive'' and ``negative'' parts. For example, one of the definitions of linear dependence over a vector space is that $\vec{0}$ can be expressed as a nontrivial linear combination of the vectors. 
\[\lambda_1\vec{v}_1 + \lambda_2\vec{v}_2+\cdots+\lambda_n\vec{v}_n = \vec{0}.\]
By rearranging such an equation we may write
\[\lambda_1\vec{v}_1+\cdots+\lambda_k\vec{v}_k = (-\lambda_{k+1})\vec{v}_{k+1} + \cdots + (-\lambda_n)\vec{v}_n.\]

Such a rearrangement is not possible over an arbitrary semiring, where additive inverses may not exist. Instead, this property is captured in the definition of linear attachment.

\begin{defi}[{\cite[p.~177]{GM08}}]
\label{def:linatt}
    Let $(V, +, \cdot)$ be a semimodule over a commutative semiring $(M, +, \times)$. A collection of vectors $\{\vec{v}_i\}_{i \in I} \subseteq V$ is said to be \demph{(linearly) attached} (negative: \demph{(linearly) detached}) if there exist disjoint subsets $A, B \subseteq I$, at most one of which is $\emptyset$, and nonzero scalars $(\lambda_{a})_{a \in A}, (\mu_{b})_{b \in B}$ such that
    \[\sum_{a \in A}\lambda_{a}\vec{v}_a = \sum_{b\in B}\mu_{b}\vec{v}_b.\]
    (The empty-sum is interpreted as the zero vector.)
\end{defi}

Our argument relies on the fact that \emph{semilinear} maps often preserve linear attachment. These maps can be defined analogously to the linear algebra case, but to my knowledge have never been studied in the context of semimodules.

\begin{defi}
    Let $\phi: (M, +, \times) \to (N, +, \times)$ be a morphism of semirings, and $(U, +, \cdot), (V, +, \cdot)$ semimodules over $M, N$ respectively. A function $T: M \to N$ is called \demph{$\phi$-semilinear} if it is:
    \begin{enumerate}[label=(\roman*)]
        \item Additive: for any $\vec{u}_1, \vec{u}_2 \in U$,
        \[T(\vec{u}_1+\vec{u}_2) = T(\vec{u}_1)+T(\vec{u}_2).\]
        \item $\phi$-homogeneous: for any $\vec{u} \in U$ and any $\lambda \in M$,
        \[T(\lambda\vec{u}) = \phi(\lambda) T(\vec{u}).\]
    \end{enumerate}
\end{defi}

Note that $\phi$-homogeneity implies $T(\vec{0}) = \vec{0}$. When $M = N$ and $\phi$ is the identity morphism, $\phi$-semilinear maps are the analogues of linear maps, they are called \demph{morphisms of semi-modules} \cite[p.~175]{GM08} or \demph{$M$-homomorphism} \cite[p.~156]{Gol99}.

\begin{eg}
    Given any semiring morphism $\phi: (M, +, \times) \to (N, +, \times)$, we obtain an induced $\phi$-semilinear map between the corresponding standard semimodules $T_\phi: M^k \to N^k$ defined componentwise $T_\phi(v_1, v_2, \ldots, v_k) = (\phi(v_1), \phi(v_2), \ldots, \phi(v_k))$.
\end{eg}

\begin{lem}
\label{lem:forward}
    Let $\phi: (M, +, \times) \to (N, +, \times)$ be an entire semiring morphism, and $T_\phi: M^k \to N^k$ the induced semilinear map. The collection $\{\vec{v}_i\}_{i \in I} \subseteq M^k$ is linearly attached, then so is $\{T_{\phi}(\vec{v}_i)\}_{i \in I} \subseteq N^k$.
\end{lem}

\begin{proof}
    Suppose the collection $\{\vec{v}_i\}_{i \in I} \subseteq M^k$ is linearly attached, so there exist disjoint subsets $A, B \subseteq I$ (not both empty) and nonzero scalars $\{\lambda_a\}_{a \in A}, \{\lambda_{b}\}_{b \in B}$ such that 
    \[\sum_{a \in A}\lambda_{a}\vec{v}_a = \sum_{b \in B}\lambda_{b\in B}\vec{v}_b.\]
    Applying $T_\phi$ to both sides of the equation we obtain
    \[\sum_{a \in A}\phi(\lambda_{a})T_{\phi}(\vec{v}_a) = \sum_{b\in B}\phi(\lambda_b)T_{\phi}(\vec{v}_b).\]
    Since $\phi$ is entire, $\{\phi(\lambda_{a})\}_{a \in A}$ and $\{\phi(\lambda_b)\}_{b \in B}$ are nonzero scalars, proving that $\{T_{\phi}(\vec{v}_i)\}_{i \in I}$ is linearly attached.
\end{proof}

In particular, for $M$ an \mCOS and $N = \B$, we have that the characteristic Boolean morphism $\vphi: M\to \B$ induces a semilinear map $T_{\vphi}: M^k \to \B^k$ which sends linearly attached collections in $M$ to linearly attached collections in $\B$. In this particular case, the converse also holds.

\begin{thm}
\label{thm:bool}
    Let $(M, +, \times)$ be an \mCOS, $\vphi: M \to \B$ the characteristic Boolean morphism, and $T: M^k \to \B^k$ the induced $\vphi$-semilinear map. The collection $\{\vec{v}_i\}_{i \in I} \subseteq M^k$ is linearly attached if and only if $\{T(\vec{v}_i)\}_{i \in I} \subseteq \B^k$ is linearly attached.
\end{thm}

\begin{proof}
    The ``only if'' direction follows from Lemma \ref{lem:forward}, so it remains to prove the ``if'' direction. Let $\{\vec{v}_i\}_{i \in I} \subseteq M^k$ be a collection of vectors such that $\{T(\vec{v}_i)\}_{i \in I} \subseteq \B^k$ is linearly attached. Then there exist disjoint subset of indices $A, B \subseteq I$ such that
    \[\sum_{a \in A}T(\vec{v}_a) = \sum_{b\in B}T(\vec{v}_b)\]
    since the only nonzero scalar of $\B$ is $1$. This is an equality in $\B^k$, so it is simply an assertion about the zero and nonzero components matching up. In detail (using the definition of $T$): for each $1 \leq i \leq k$, there exists some $a\in A$ such that the $i$-th component of $\vec{v}_a$ is nonzero if and only if there exists some $b \in B$ such that the $i$-th component of $\vec{v}_b$ is nonzero. Therefore,
    \[\sum_{a\in A}\m\vec{v}_a = \sum_{b\in B}\m\vec{v}_b\]
    so that $\{\vec{v}_i\}_{i \in I}$ is linearly attached.
\end{proof}

\begin{rk}
    The dioid $\N$ is not an \mCOS, but since it embeds into the field $\Q$, the same proof works: instead of multiplying by $\m$, we multiply by the common denominator to transfer a linear attachment in $\Q^k$ to a linear attachment in $\N^k$.
\end{rk}

We see that for \mCOS{s}, questions of linear attachment and detachment are completely equivalent to the corresponding questions for the much simpler semiring $\B$. Linear attachment in $\B$ can sometimes be detected by linear dependence over $\R$.

\begin{thm}
\label{thm:toR}
    If the collection $\{\vec{v}_i\}_{i \in I} \subseteq \B^k$ is linearly dependent when viewed as vectors in $\R^k$, then it is linearly attached in $\B^k$.
\end{thm}

\begin{proof}
    We may assume without loss of generality that $\vec{0}$ is not in the collection (for otherwise it is trivially linearly attached). To avoid confusion, we use $\{\vec{v}_i\}_{i \in I}$ for the collection of vector in $\B^k$ and $\{\widetilde{\vec{v}}_i\}_{i \in I}$ for the same collection when viewed as vectors in $\R^k$.

    Suppose $\{\widetilde{\vec{v}}_i\}_{i \in I}$ is linearly dependent, so there is a finite subset $J \subseteq I$ and nonzero real scalars $\{\lambda_j\}_{j \in J}\subseteq \R$ such that
    \[\sum_{j \in J}\lambda_j\widetilde{\vec{v}}_j = \vec{0}.\]
    Let $A, B \subseteq J$ be the collections of indices for which the scalars are positive or negative, respectively:
    \begin{align*}
        &A = \st{j \in J}{\lambda_j > 0}, &&B = \st{j \in J}{\lambda_j<0}.
    \end{align*}
    Note that $\{A, B\}$ is a partition of $J$: $A\union B = J$ (since the $\lambda_j$ are nonzero), $A \cap B = \emptyset$, and $A, B \neq \emptyset$ (because the components of the $\widetilde{\vec{v}}_i$ are all nonnegative, in fact in $\{0, 1\}$). We may therefore rewrite the linear dependence relation as
    \[\sum_{a \in A}\lambda_a\widetilde{\vec{v}}_a = \sum_{b \in B}(-\lambda_b)\widetilde{\vec{v}}_b.\]
    Every vector appearing in this equality has positive components. In particular, for each $1 \leq i \leq k$, there exists some $a\in A$ such that the $i$-th component of $\vec{v}_a$ is nonzero if and only if there exists some $b \in B$ such that the $i$-th component of $\vec{v}_b$ is nonzero. Therefore,
    \[\sum_{a\in A}\vec{v}_a = \sum_{b\in B}\vec{v}_b.\]
    proving that $\{\vec{v}_i\}_{i \in I}$ is linearly attached.
\end{proof}

\begin{eg}
\label{eg:noconv}
    The vectors
    \[\begin{mx}
        1 \\ 0 \\ 1 \\ 0
    \end{mx}\stp;\stp \begin{mx}
        0 \\ 1 \\ 0 \\ 1
    \end{mx}\stp;\stp \begin{mx}
        1 \\ 1 \\ 1 \\ 0
    \end{mx}\stp;\stp \begin{mx}
        0 \\ 1 \\ 1 \\ 1
    \end{mx}\]
    are clearly linearly attached in $\B^4$ (e.g., the sum of the first two is the same as the sum of the last two). However, they are linearly independent over $\R$:
    \[\det\begin{mx}
        1 & 0 & 1 & 0 \\
        0 & 1 & 1 & 1 \\
        1 & 0 & 1 & 1 \\
        0 & 1 & 0 & 1
    \end{mx} = 1.\]
    This shows that the converse to Theorem \ref{thm:toR} fails.
\end{eg}

\begin{crl}
\label{crl:max}
    Let $(M, +, \times)$ be an \mCOS. Any collection of $k+1$ (or more) vectors in $M^k$ is linearly attached.
\end{crl}

\begin{proof}
    Let $\vec{v}_1, \ldots, \vec{v}_{k+1} \in M^{k+1}$ be a collection of $k+1$ vectors. Let $T: M^k \to \B^k$ be the $\vphi$-semilinear map induced by the characteristic Boolean morphism $\vphi: M \to \B$. By Theorem \ref{thm:bool}, $\vec{v}_1, \ldots, \vec{v}_{k+1}$ are linearly attached if and only if $T(\vec{v}_1), \ldots, T(\vec{v}_{k+1})$ are linearly attached. Viewed as $k+1$ vectors in $\R^k$, we know that $T(\vec{v}_1), \ldots, T(\vec{v}_{k+1})$ must be linearly dependent, so by Theorem \ref{thm:toR}, $T(\vec{v}_1), \ldots, T(\vec{v}_{k+1})$ are linearly attached.
\end{proof}

This solves Puzzle \ref{puz:howlarge}: the maximal size of a family $\ffam \subseteq \pset{\{1, 2, \ldots, n\}}$ with the disparate union property is $n$.

\begin{crl}[Solution to Puzzle \ref{puz:howlarge}]
\label{crl:solhowlarge}
    Let $\ffam \subseteq \pset{\{1, 2, \ldots, n\}}$ be a family with $\crd{\ffam}\geq n+1$. Then there exist two disjoint subfamilies $\mcal{A}, \mcal{B} \subseteq \ffam$ (not both empty) such that $\Union\mcal{A} = \Union \mcal{B}$. 
\end{crl}

\begin{proof}
    To each set in $\ffam = \{S_1, S_2, \ldots, S_{k}\}$ assign its incidence vector in $\B^n$: $\vec{v}_1, \vec{v}_2, \ldots, \vec{v}_k$. Since $\crd{\ffam} = k \geq n+1$, Corollary \ref{crl:max} implies that these vectors are linearly attached, so there exist disjoint subsets of indices $A, B \subseteq \{1, 2, \ldots, k\}$ (not both empty) such that 
    \[\sum_{a \in A}\vec{v}_a = \sum_{b \in B}\vec{v}_b\]
    (the only nonzero scalar in $\B$ is $1$). Addition of incidence vectors in $\B$ exactly corresponds to union of the sets they represent.
\end{proof}

The very same reasoning also presents a partial solution to Puzzle \ref{puz:criterion}.

\begin{crl}[Partial Solution to Puzzle \ref{puz:criterion}]
\label{crl:part}
    Let $\ffam \subseteq \pset{\{1, 2, \ldots, n\}}$ be a family of size $n$, $\crd{\ffam} = n$. If the matrix of incidence vectors has determinant $0$ as a real matrix, then there exist two disjoint subfamilies $\mcal{A}, \mcal{B} \subseteq \ffam$ (not both empty) such that $\Union\mcal{A} = \Union \mcal{B}$.
\end{crl}

Example \ref{eg:noconv} shows that this condition is sufficient but not necessary. In Section \ref{sec:bid} we present a necessary and sufficient criterion in terms of the bideterminant (rather than the determinant) of the matrix of incidence vectors. We conclude this section with a partial converse to Theorem \ref{thm:toR} which will play an important role in our investigation of the bideterminant and our solution to Puzzle \ref{puz:criterion}: we (constructively) prove that linearly attached Boolean vectors in $\B^k$ arise as the image of linearly attached vectors in $\N^k$. Linearly attached vectors in $\N^k$ are easily seen to be linearly dependent when regarded as vectors in $\R^k$. Note that $\N$ is an \COS which is not an \mCOS; nevertheless, since it is zerosumfree and entire, the characteristic Boolean morphism $\vphi: \N \to \B$ is well-defined (cf.~Lemma \ref{lem:boolmorph}).

\begin{thm}
\label{thm:fromN}
Let $\{\vec{v}_i\}_{i\in I}\subseteq \B^k$ be a collection of linearly attached Boolean vectors. Then there exists a collection $\{\vec{u}_i\}_{i \in I} \subseteq \N^k$ of linearly attached natural vectors such that $\vec{v}_i = T(\vec{u}_i)$, where $T: \N^k \to \B^k$ is the $\vphi$-semilinear map induced by the characteristic Boolean morphism $\vphi: \N \to \B$.
\end{thm}

\begin{proof}
    We may assume without loss of generality that $\vec{0} \notin \{\vec{v}_i\}_{i \in I}$. For $1 \leq j \leq k$, we use the notation $(\vec{w})_j$ for the $j$-th component of the vector $\vec{w}$.

    Let $A, B \subseteq I$ be disjoint subsets of indices such that
    \begin{equation}
    \tag{$*$}
    \label{eq:booleq}
    \sum_{a \in A}\vec{v}_a = \sum_{b \in B}\vec{v}_b.
    \end{equation}
    We denote by $\alpha_j, \beta_j$ the number of vectors with nonzero $j$-th component in each of $A$ and $B$:
    \begin{align*}
        &\alpha_j \eqdef \crd{\st{(\vec{v}_a)_j = 1}{a \in A}},
        &&\beta_j \eqdef \crd{\st{(\vec{v}_b)_j = 1}{b \in B}}.
    \end{align*}
    Equation \eqref{eq:booleq} guarantees that these numbers vanish together: either $\alpha_j = \beta_j = 0$ or $\alpha_j, \beta_j > 0$. We denote by $\gamma_j$ their difference:
    \[\gamma_j \eqdef \max\{\alpha_j, \beta_j\} - \min\{\alpha_j, \beta_j\}.\]
    We now show how to construct $\{\vec{u}_i\}_{i \in I} \subseteq \N^k$. Start by initializing $\vec{u}_i = \vec{v}_i$ for all $i \in I$. For each $1 \leq j \leq k$ do:
    \begin{enumerate}[label=(\roman*)]
        \item if $\alpha_j = \beta_j$, move to the next component;
        \item if $\alpha_j > \beta_j$, choose any $b \in B$ with $(\vec{v}_b)_j = 1$ (one such must exist) and define $(\vec{u}_b)_j = \gamma_j+1$, then move to the next component;
        \item if $\alpha_j < \beta_j$, choose any $a \in A$ with $(\vec{v}_a)_j = 1$ (one such must exist) and define $(\vec{u}_a)_j = \gamma_j+1$, then move to the next component.
    \end{enumerate}
    It is clear that at the end of this procedure $T(\vec{u}_i) = \vec{v}_i$ for all $i \in I$, since we have not changed any nonzero component to zero or vice versa. Moreover, we claim that the following equality holds in $\N^k$:
    \[\sum_{a \in A}\vec{u}_a = \sum_{b \in B}\vec{u}_b.\]
    Indeed, we can see it holds for each component $1 \leq j \leq k$. Note that, by the definition of $\alpha_j$ and $\beta_j$, when the sums are computed in $\N$ we have
    \begin{align*}
        &\sum_{a\in A}(\vec{v}_a)_j = \alpha_j,
        &&\sum_{b\in B}(\vec{v}_b)_j = \beta_j.
    \end{align*}
    \begin{enumerate}[label=(\roman*)]
        \item if $\alpha_j = \beta_j$, the $j$-th component of the $\vec{u}_i$'s is identical to that of the $\vec{v}_i$'s, so that 
        \[\sum_{a \in A}(\vec{u}_a)_j = \alpha_j = \beta_j = \sum_{b \in B}(\vec{u}_b)_j.\]
        \item if $\alpha_j > \beta_j$, then the $j$-th component of the vectors $\vec{u}_a$ (with $a \in A$) is identical to that of $\vec{v}_a$ (with $a \in A$) so that
        \[\sum_{a \in A}(\vec{u}_a)_j = \alpha_j.\]
        On the other hand, the $j$-th component of the vectors $\vec{u}_b$ (with $b \in B$) is identical to that of $\vec{v}_{b}$ (with $b \in B$), except for one vector whose component increased by $\gamma_j$. Therefore,
        \[\sum_{b \in B}(\vec{u}_b)_j = \beta_j+\gamma_j = \alpha_j.\]
        \item If $\alpha_j < \beta_j$ then completely analogously to the previous case we have
        \[\sum_{b \in B}(\vec{u}_b)_j = \beta_j = \alpha_j+\gamma_j = \sum_{a \in A}(\vec{u}_a)_j.\]
    \end{enumerate}
\end{proof}

\subsection{Interlude: Unpacking the Proof}
\label{subsec:unpack}
The solution to the puzzle presented as Corollary \ref{crl:solhowlarge} can be ``translated'' to a simple linear-algebraic solution by unpacking all the theorems and definitions having to do with dioids, as follows.

\begin{proof}[``Translated'' Solution to Puzzle \ref{puz:howlarge}]
Let $\ffam \subseteq \pset{\{1, 2, \ldots, n\}}$ be a family with $\crd{\ffam}\geq n+1$. To each set in $\ffam = \{S_1, S_2, \ldots, S_{k}\}$ assign its incidence vector $\vec{v}_1, \ldots, \vec{v}_k$; since $k\geq n+1$, the resulting collection of vectors must be linearly dependent in $\R^n$. Then there exists a nontrivial linear combination of $\vec{0}$:
\[\lambda_{1}\vec{v}_1 + \cdots + \lambda_{k}\vec{v}_k = \vec{0}\]
and we may assume without loss of generality that none of the scalars are $0$. Let $A$ be the set of indices corresponding to positive scalars, and $B$ the set of indices corresponding to negative scalars.
\[\sum_{a \in A}\lambda_{a}\vec{v}_a = \sum_{b \in B}(-\lambda_b)\vec{v}_b\]
with all scalar coefficients positive. Then $\mcal{A} \eqdef \st{S_a}{a \in A}$ and $\mcal{B} \eqdef \st{S_b}{b\in B}$ are disjoint subfamilies of $\ffam$ (not both empty) which have the same union, since $\Union\mcal{A}$ and $\Union\mcal{B}$ have the same incidence vector (with $i$-th component $0$ if and only if the $i$-th component of $\sum_{a\in A}\lambda_a\vec{v}_a = \sum_{b \in B}(-\lambda_b)\vec{v}_b$ is $0$).
\end{proof}

The solution is both short and elementary, however it requires at least three distinct insights. First, it is not apparent that it is profitable to work with incidence vectors over $\R$. While translating sets to incidence vectors is by now a standard technique in extremal combinatorics, it is not apparent that this technique is applicable to this problem---because the union operation does not correspond in any clear way to field operations, as explained in Section \ref{sec:motivation} above. It is also not apparent that the field should be chosen to be infinite rather than finite.

Second, the idea of separating the positive and negative scalars is very natural in the context of semirings, but not often encountered in general linear algebra, apart perhaps from investigations into various positivity conditions.

Finally, the idea that only the vanishing or nonvanishing of the coordinates matter, rather than the individual components is natural when working with ``linearity'' relations over $\B$ or over an \mCOS, but not in general linear algebra, apart perhaps from investigations into combinatorial classes of matrices.

The setting of semirings, by limiting the supply of available mathematical tools, narrows the search space of possible solutions. Techniques which are natural in the setting of semirings seem more arbitrary in the setting of general linear spaces. This is even more apparent with the solution to Puzzle \ref{puz:criterion} (presented in the next section), which characterizes all maximal disparate-union families; I do not know of a simple ``translation'' of that solution to the setting of vector spaces over fields.

\section{Bideterminant}\label{sec:bid}
Matrix addition and matrix multiplication are both well-defined for matrices with entries in an arbitrary commutative semiring, since both of these operations are defined in terms of the underlying addition and multiplication operations. On the other hand, the determinant uses the additive-inverses in its definition, and so cannot be straightforwardly defined for matrices over semirings. As a workaround, we keep track of the positive and negative parts of the determinant separately, through an object known as the bideterminant.

\begin{defi}[{\cite[p.~214]{Gol99}, \cite[p.~182]{GM08}}]
    Let $A = [a_{ij}]$ be an $n\times n$ square matrix with entries in a commutative semiring $(M, +, \times)$. Denote by $S_n^+$ and $S_n^-$ the even and odd permutations of the symmetric group, respectively (i.e., $S_n^+ = A_n$, the alternating group and $S_n^- = S_n\setminus A_n$). Then the \demph{positive determinant} $\det^+A$ and \demph{negative determinant} $\det^-A$ are defined as
    \begin{align*}
        &\mbox{$\det^+A$} = \sum_{\sigma\in S_n^+}\prod_{i=1}^{n}a_{i\sigma(i)},
        &&\mbox{$\det^-A$} = \sum_{\sigma \in S_n^-}\prod_{i=1}^{n}{a_{i\sigma(i)}}.
    \end{align*}
    (All operations should be performed in $M$.) The pair $(\det^+A, \det^-A)$ is called the \demph{bideterminant} of $A$.
\end{defi}

For convenience, Gondran and Minoux call the summand
\[w_\sigma(A) \eqdef \prod_{i=1}^{n}a_{i\sigma(i)}\]
the \demph{weight} of the permutation $\sigma$ in the matrix $A$. In the literature on nonnegative matrices, this is known as the \demph{(generalized) diagonal} of the matrix $A$ associated with the permutation $\sigma$, and we shall adopt this latter terminology.

The bideterminant shares many of the properties of the determinant. The following properties are all straightforward consequences of the definition above in terms of permutations and we omit the proofs.
\begin{prop}[Elementary properties of the bideterminant]
Let $(M, +, \times)$ be a commutative semiring, for any $n\times n$ matrix $A$ with entries in $M$ and bideterminant $(\det^+A, \det^-A)$:
\begin{enumerate}[label=(\roman*)]
\label{prop:bid}
    \item The bideterminant is \demph{alternating}: if two rows are swapped in $A$, the resulting matrix has bideterminant $(\det^-A, \det^+A)$.
    \item The bideterminant is \demph{multilinear}:
    \begin{itemize}
        \item if a row of $A$ is multiplied by a scalar $\lambda\in M$, then the resulting matrix has bideterminant $\lambda(\det^+A, \det^-A)$;
        \item if a vector $\vec{v}$ is added to the $i$-th row of $A$, then the resulting matrix has bideterminant $(\det^+A, \det^-A) + (\det^+B, \det^-B)$, where $B$ is the matrix which is identical to $A$ except for having $\vec{v}$ as its $i$-th row.
    \end{itemize} 
    \item The bideterminant of the identity matrix $I$ is $(\det^+I, \det^-I) = (1, 0)$.
    \item The bideterminant of a matrix with a row of $0$'s is $(0, 0)$.
    \item The bideterminant of the transpose $A^T$ is the same as the bideterminant of $A$:\\ $(\det^+A^T, \det^-A^T) = (\det^+A, \det^-A)$.
    \item For $A = [a_{ij}]$ upper-triangular (i.e., $a_{ij} = 0$ for $i>j$), the bideterminant is $(\det^+A, \det^-A) = (\prod_{i=1}^{n}a_{ii}, 0)$.
    \item\label{it:mor} Let $\phi: (M, +, \times) \to (N, +, \times)$ be a semiring morphism, and denote by $\phi(A)$ the matrix resulting from applying $\phi$ to each entry of $A$. Then $(\det^+\phi(A), \det^-\phi(A)) = (\phi(\det^+A), \phi(\det^-A))$.
\end{enumerate}
\end{prop}

We are interested in the relationship between the bideterminant and linear attachment in the standard semimodule $M^n$. For square matrices over a field, it is a basic theorem of linear algebra that the columns (equivalently, rows) are linearly dependent if and only if the determinant is $0$. In terms of the bideterminant, this would mean $\det^+A = \det^-A$. This equality, however, does \emph{not} hold for general commutative semirings or even dioids.

When both $(M, +)$ and $(M\setminus\{0\}, \times)$ are cancellative, the semiring can be embedded into a commutative ring, and so the equality holds. A more sophisticated class of semirings for which equality also holds is a generalization of tropical rings. A semiring $(M, +, \times)$ is called \demph{selective}\footnote{This is Gondran and Minoux's terminology \cite{GM08}, Golan \cite[p.~228]{Gol99} calls such rings \demph{extremal}.} if $a+b \in \{a, b\}$ for all $a, b \in M$. Every selective semiring is also a dioid, the canonical preorder being a total order. A \demph{selective-invertible dioid} is a selective dioid $(M, +, \times)$ such that $(M\setminus\{0\}, \times)$ is a group (an example of such a dioid is min-plus on the nonnegative reals $(\R^+\union\{\infty\}, \min, +)$). By associating with a matrix a certain complete bipartite graph and carefully studying its perfect matchings and cycles, Gondran and Minoux \cite{GM77,GM78} (reproduced in \cite[ch.~5]{GM08}) proved with much ingenuity that a matrix with entries in a selective-invertible dioid has linearly-attached columns if and only if $\det^+A = \det^-A$. Part of their proof contains a more general unidirectional result for matrices with entries in a selective dioid (which is not necessarily selective-invertible): if its columns satisfy a restricted (all scalars are $1$) linear attachment relation, then $\det^+A = \det^-A$.

The Boolean dioid $\B$ is trivially selective-invertible (note that $(\B\setminus\{0\}, \times)$ is the trivial group), so Gondran and Minoux's Theorem \cite{GM77,GM78} applies: the columns of $A$ are linearly attached if and only if $\det^+A= \det^-A$. We now offer a much simpler non-graph-theoretic proof of this result for $\B$.

\begin{thm}
\label{thm:boolbid}
    Let $A$ be an $n\times n$ matrix with entries in $\B$. The columns of $A$ are linearly attached as vectors in $\B^n$ if and only if $\det^+A = \det^-A$.
\end{thm}

For convenience, we separate the proof into the forward and backward directions. To avoid confusion, we use subscripts to indicate the semiring over which the determinant is taken.

\begin{proof}[Proof of the ``only if'' direction]
    Suppose the columns of $A$ are linearly attached as vectors in $\B^n$. By Theorem \ref{thm:fromN}, there exists an $n\times n$ matrix $N$ with entries in $\N$ such that $\vphi(N) = A$, where $\vphi: \N \to \B$ is the characteristic Boolean morphism (and $\vphi(N)$ is interpreted as in Proposition \ref{prop:bid}\ref{it:mor}, by applying $\vphi$ to each entry of the matrix), and the rows of $N$ are linearly attached (as vectors in $\N^n$). Now, the columns of $N$ are linearly dependent as vectors in $\R^n$, so that $\det_{\R}N = 0$. By the definition of the determinant, this means that $\det_{\R}^+N = \det_{\R}^-N$, which in turn means $\det_{\N}^+N = \det_{\N}^-N$. By Proposition \ref{prop:bid}\ref{it:mor} we have
    \[
        \mbox{$\det_{\B}^{\pm}{A}$} = \mbox{$\det_{\B}^{\pm}{\vphi(N)}$} = \mbox{$\vphi(\det_{\N}^{\pm}N)$}
    \]
    from which we conclude $\det^+_{\B}A = \det_{\B}^-A$.
\end{proof}

\begin{proof}[Proof of the ``if'' direction]
    Suppose $\det^+A = \det^-A$. Note that if $\det_{\R}A = 0$, then the columns of $A$ are linearly dependent as vectors in $\R^n$, and from Theorem \ref{thm:toR} we may conclude that the columns of $A$ are linearly attached as vectors in $\B^n$. Assume therefore that $\det_{\R}A \neq 0$, and without loss of generality that $\det_{\R}^+A > \det_{\R}^-A$ (otherwise, interchange two columns of $A$). Since the entries of $A$ are either $0$ or $1$, we conclude that $\det_{\R}^+ > 0$ which implies $\det_{\B}^+A = 1$. Since $\det_{\B}^+A = \det_{\B}^-A$, we conclude that $\det_{\B}^-A = 1$ so that $\det_{\R}^+A > 0$. To summarize,
    \[\mbox{$\det_{\R}^+A$} > \mbox{$\det_{\R}^-A$} > 0.\]

    Let $(\R^+, +, \times)$ be the dioid of nonnegative real numbers, which we note is zerosumfree and entire so that the characteristic Boolean morphism $\vphi: \R^+ \to \B$ is well-defined. We prove that there exists a matrix $R$ with entries in $\R^+$ such that $\vphi(R) = A$ and whose columns are linearly attached as vectors in $(\R^+)^n$. By Lemma \ref{lem:forward}, this implies that the columns of $A$ are linearly attached as vectors in $\B^n$.

    To see that such a matrix $R$ must exist, let $\sigma \in S_n^-$ be an odd permutation with a nonzero generalized diagonal: 
    \[w_{\sigma}(A) = \prod_{i=1}^{n}a_{i\sigma(i)} = 1.\] 
    Such a permutation must exist since $\det_{\R}^-(A) \neq 0$ and the entries of $A$ are either $0$ or $1$. This means that each $a_{i\sigma(i)}$ is $1$. Let $R(x)$ be the matrix obtained from $A$ by changing each $a_{i\sigma(i)}$ to a variable $x$, and consider the polynomial $p(x) \eqdef \det_{\R}R(x)$.

    For $x = 1$, we have $R(x) = A$ so that $p(1) > 0$. On the other hand, $p(x) = \det_{\R}^+R(x) - \det_{\R}^-(x)$ contains the term $-x^n$ and every positive term is of order at most $x^{n-2}$. Therefore, $\lim_{x\to\infty}p(x) = -\infty$, so by the continuity of the (real) determinant, there exists some $a \in (1, \infty)$ with $p(a) = 0$.

    We have $\vphi(R(a)) = A$ and the columns of $R(a)$ are linearly dependent as vectors over $\R^n$; so $\vec{0}$ is a nontrivial $\R$-linear combination of the columns of $R(a)$. Separating the positive and negative scalar coefficients we obtain an $\R^+$-linear attachment of the columns of $R(a)$. 
\end{proof}

As an immediate corollary, we obtain a criterion for the disparate union property and a solution to Puzzle \ref{puz:criterion}.

\begin{crl}[Solution to Puzzle \ref{puz:criterion}]
\label{crl:solcrit}
    A family $\ffam \subseteq \pset{\{1, 2, \ldots, n\}}$ of size $n$ has the disparate union property if and only if the bideterminant (computed in $\B$) of the matrix of incidence vectors is $(1, 0)$ or $(0, 1)$.
\end{crl}

Continuing our discussion from Section \ref{subsec:unpack}, it appears that this criterion is not easy to directly ``translate'' to linear algebra, because it compares the positive and negative parts of the determinant (corresponding to even and odd permutations). In a vector space over a field, both the determinant and the permanent can be computed from the bideterminant as $\det A = \det^+ A - \det^-A$ and $\per A = \det^+ A + \det^-A$. However, while computing the determinant over a field is simple to do efficiently (e.g., by Gaussian elimination), it is conjectured that computing the permanent is intractable in general. In a celebrated paper, L.~Valiant defined the complexity class $\#\mathsf{P}$, which is at least as hard as $\mathsf{NP}$, and proved that computing (over $\R$ or $\C$) the permanent of a matrix with entries in $\{0, 1\}$ is $\#\mathsf{P}$-complete \cite{Val79}. In contrast, computing the permanent over dioids is sometimes easy while computing the determinant is sometimes $\mathsf{NP}$-complete; Minoux \cite{Min82} (in an unpublished note, the argument is reproduced in \cite[pp.~204--205]{GM08}), noted one such example in a Min-Max algebra. This naturally raises the question of whether the criterion presented in Corollary \ref{crl:solcrit} can be computed efficiently. We now argue that this is indeed the case. 

Butkovi\v{c} \cite{But95} calls a \demph{min-algebra} a triple $(G, \min, \times)$ where $(G, \times)$ is a linearly-ordered commutative group. An example of such a min-algebra is $(\R, \min, +)$. Butkovi\v{c} proved \cite{But95} that a matrix $A$ with with entries in a min-algebra has $\det^+A = \det^-A$ if and only if a certain associated directed graph has an even (directed) cycle (also called an \emph{even dicycle} or an \emph{even circuit}), and that therefore the problem of deciding whether $\det^+A = \det^-A$ for that class of matrices is polynomial-time reducible to the problem of deciding whether a given directed graph has an even cycle. The even-cycle problem, along with multiple other problems which polynomial-time reduce to it, was later proved to be polynomial-time solvable in a famous paper of Robertson, Seymour, and Thomas \cite{RST99}.

We now show how to use Butkovi\v{c}'s Theorem \cite{But95} to check our criterion from Corollary \ref{crl:solcrit}. Let $A$ be a matrix with entries in $\B$. To check whether $\det^+A = \det^-A$:
\begin{itemize}
    \item Check whether $\det_{\R}A = 0$. If so, we are done.
    
    \item Otherwise, $\det_{\R}A \neq 0$, so that the bideterminant is one of $(0, 1)$, $(1, 0)$, or $(1, 1)$. Then $\det^+A = \det^-A$ if and only if each of $S_n^+$ and $S_n^-$ has a permutation $\sigma$ such that the associated generalized diagonal is nonzero $w_{\sigma}(A) = 1$.
    
    \item Let $A^c$ be the matrix obtained from $A$ by switching all the $0$ entries to $1$ and all the $1$ entries to $0$, i.e., the $(i,j)$-entry of $A^c$ is $1$ if and only if the $(i, j)$-entry of $A$ is $0$. We shall consider $A^c$ as a matrix with entries in $(\R, \min, +)$.

    Note that $\det^+A = \det^-A$ if and only if each of $S_n^+$ and $S_n^-$ has a permutation such that the associated generalized diagonal in $A^c$ is $0$ when computed in $(\R, \min, +)$. By the definition of $A^c$ (and the assumptions on $A$), this is the case if and only if $\det^+A^c = \det^-A^c$ when computed in $(\R, \min, +)$, which completes the reduction. 
\end{itemize}

We conclude the paper with a corollary applying our results from Section \ref{sec:lin} to obtain a criterion for linear attachment over \mCOS{s}. Note that an \mCOS may or may not be selective, and (excepting $\B$) is never selective-invertible, so Gondran and Minoux's Theorem does not apply to such dioids. A consequence of their unidirectional theorem for selective dioids is that if the columns of a matrix (with entries in a selective dioid $M$) are linearly attached, then there must exist some scalar $\lambda \in M$ such that $\lambda\det^+ A = \lambda\det^-A$. In an \mCOS, an equation of this form means that $\det^+A, \det^-A$ are zero or nonzero together (since we can always take $\lambda = \m$). We prove that this condition is necessary and sufficient for linear attachment over an \mCOS.

\begin{crl}
\label{crl:mbid}
    Let $A$ be an $n\times n$ matrix with entries in an \mCOS $(M, +, \times)$. Then the columns of $A$ are linearly attached if and only if $\det^+A, \det^-A$ are both zero or both nonzero.
\end{crl}

\begin{proof}
    Let $\vphi: M \to \B$ be the characteristic Boolean morphism. By Theorem \ref{thm:bool}, the columns of $A$ are linearly attached in $M^n$ if and only if the columns of $\vphi(A)$ are linearly attached in $\B^n$. By Theorem \ref{thm:boolbid}, the columns of $\vphi(A)$ are linearly attached if and only if $\det_{\B}^+\vphi(A) = \det_{\B}^-\vphi(A)$. Applying Proposition \ref{prop:bid}\ref{it:mor}, we find that the columns of $M$ are linearly attached in $M^n$ if and only if
    \[\mbox{$\vphi(\det_{M}^+A)$} = \mbox{$\det_{\B}^+\vphi(A)$} = \mbox{$\det_{\B}^-\vphi(A)$} = \mbox{$\vphi(\det_M^-A)$}.\]
    By the definition of $\vphi$ we conclude that the columns of $M^n$ are linearly attached if and only if $\det_{M}^+ A, \det_{M}^-A$ are both zero or both nonzero.
\end{proof}

\section{Conclusion}
Finding a suitable algebraic model for a given problem is a common proof strategy throughout mathematics. In this paper we defined the algebraic structure of \mCOS{s} (Definition \ref{def:mCOS}) and suggest it is a useful setting for combinatorial questions involving idempotent operations on families of finite sets (such as union). Any commutative monoid $(M, \times)$ can be embedded in an \mCOS via LaGrassa's construction \cite{LaG95}, and if the monoid is finite the powerset construction results in an \mCOS (Lemma \ref{lem:powerdioid}).

We have shown that a generalization of linear independence called \emph{linear attachment} (Definition \ref{def:linatt}) is particularly easy to detect in moduloids over \mCOS{s}, reducing to detecting linear attachment over the Boolean dioid $\B$ (Theorem \ref{thm:bool}). Linear attachment over $\B$ can in turn be detected by the bideterminant (Theorem \ref{thm:boolbid}), and this can be done efficiently via Butkovi\v{c}'s algorithm \cite{But95} by solving the even-cycle problem. Taken together, these results allow us to adapt \emph{dimension arguments} in extremal combinatorics to the setting of moduloids over \mCOS{s}, answering Puzzle 1 (Corollary \ref{crl:solhowlarge}) and Puzzle 2 (Corollary \ref{crl:solcrit}) with which we started the paper.

Golan's monograph \cite{Gol99} and Gondran and Minoux's monograph \cite{GM08} showcase many applications of semiring theory in general (the former) and dioids in particular (the latter), mainly from the fields of optimization and operations research; it is natural to explore the use of these algebraic structures in addressing purely combinatorial questions such as those arising in extremal combinatorics. There is an intriguing possibility that such structures may help us surmount well-understood obstructions to applying commonly used tools---such as the requirement of working with a vector space over a field. The author hopes that future research may identify more examples of such situations.

\section*{Acknowledgements}
I thank my sister, Ofek Gross, for suggesting the ``distinct union property'' puzzle upon hearing the ``disparate union property'' puzzle. I am grateful to Almut Burchard, Shubhanghi Saraf, and Swastik Kopparty for their encouragement.

\newpage
\bibliographystyle{alpha}
\bibliography{ms}
\end{document}